\documentclass[10pt, 14paper,reqno]{amsart}
\usepackage{amsmath,amsfonts,amssymb}
\usepackage{longtable}
\usepackage{color}
\makeatletter
\theoremstyle{plain}
\newtheorem{theorem}{Theorem}
\newtheorem{lemma}{Lemma}
\newtheorem{corollary}{Corollary}
\theoremstyle{proof}
\theoremstyle{definition}

\theoremstyle{remark}

\theoremstyle{lamma}

\numberwithin{equation}{section}
\numberwithin{lemma}{section}
\numberwithin{theorem}{section}
\numberwithin{remark}{section}
\numberwithin{prop}{section}
\numberwithin{corollary}{section}

\usepackage{amsmath}
\usepackage{amsfonts}   
\usepackage{amssymb}
\usepackage{amssymb, amsmath, amsthm}
\usepackage[breaklinks]{hyperref}

\theoremstyle{thmrm}

\numberwithin{conjecture}{section}

\begin{document}
\title[Cyoclotomic numbers of order $2\ell^{2}$]{\large{Cyoclotomic numbers of order $2\ell^{2}$ with prime $\ell$}}
\author{Md Helal Ahmed, Jagmohan Tanti and Azizul Hoque}
\address{Md Helal Ahmed @ Centre for Applied Mathematics, Central University of Jharkhand, Ranchi-835205, India}
\email{ahmed.helal@cuj.ac.in}

\address{Jagmohan Tanti @ Centre for Applied Mathematics, Central University of Jharkhand, Ranchi-835205, India}
\email{jagmohan.t@gmail.com}

\address{Azizul Hoque @Azizul Hoque, Harish-Chandra Research Institute,
Chhatnag Road, Jhunsi,  Allahabad 211 019, India.}
\email{ azizulhoque@hri.res.in}

\maketitle
\begin{abstract}
The problem of determining cyclotomic numbers in terms of the solutions of certain Diophantine systems has been treated by many authors since the age of 
Gauss. In this paper we obtain an explicit expression for cyclotomic numbers of order $2\ell^{2}$ in terms of the coefficients of the Jacobi sums of 
lower orders. At the end, we illustrate the nature of two matrices corresponding to two types of cyclotomic numbers.
\end{abstract}
\section{Introduction}
Let $p$ be an odd prime integer and $q=p^r$ with $r\geq 1$ an integer. Let $e$ be a non-trivial divisor of $q-1$. Then $q=ek+1$ for some positive integer $k$. 
Suppose $\gamma$ is a generator of the cyclic group $\mathbb{F}^{\times}_{q}$. For a primitive $e$-th root $\zeta_e$ of unity, define a multiplicative 
character $\chi_e$ on $\mathbb{F}^{\times}_{q}$ by $\chi_e(\gamma)=\zeta_e$.
For $0\leq i, j\leq e-1$, the cyclotomic numbers $(i, j)_e$ of order $e$ are defined as follows:
$$
(i,j)_e:=\#\{v\in\mathbb{F}_q\setminus \{0,-1\}\mid {\rm ind}_{\gamma}v\equiv a \pmod e ,  {\ \rm ind}_{\gamma}(v+1)\equiv b \pmod e\}.
$$
We now extent $\chi_e$ to a map from $\mathbb{F}_q$ to $\mathbb{Q}({\zeta_e})$ by taking $\chi_e(0)=0$. The Jacobi sums of order $e$ is defined by 
$$
J_e(i,j)= \sum_{v\in \mathbb{F}_q} \chi_e^i(v) \chi_e^j(v+1).
$$
The cyclotomic numbers $(i,j)_e$ and the Jacobi sums $J_e(i,j)$ are well connected  by the following relations:
\begin{equation} \label{01}
 \sum_i\sum_j(i,j)_e\zeta_e^{ai+bj}=J_e(a,b),
 \end{equation} 
 and
 \begin{equation} \label{00}
 \sum_i\sum_j\zeta_e^{-(ai+bj)}J_e(i,j)=e^{2}(a,b)_e.
 \end{equation} 
\eqref{01} and \eqref{00} show that if we want to calculate all the cyclotomic numbers $(i,j)_e$ of order $e$, it is sufficient to calculate all the 
Jacobi sums $J_e(i,j)$ of the same order, and vice-versa. 

Cyclotomic numbers are one of the most important objects in number theory and in other branches of mathematics. These number have been extensively used in
coding theory, cryptography and in other branches of information theory. One of the central problems in the study of these numbers is the determination of
all cyclotomic numbers of a specific order for a given field in terms of solutions of certain Diophantine system.
This problem has been treated by many mathematicians including C. F. Gauss who had determined all the cyclotomic numbers of order $3$ in the field 
$\mathbb{F}_q$ with prime $q\equiv 1\pmod 3$.  Cyclotomic numbers of order $e$ over the field $\mathbb{F}_q$ with characteristic $p$, in general, can not 
be determined only in terms of $p$ and $e$, but that one requires a quadratic partition of $q$ too.

Complete solutions to this cyclotomic number problem have been computed for some specific orders. For instance, the cyclotomic numbers of prime order $e$ 
in the finite field $\mathbb{F}_q$ with $q=p^r$ and $p\equiv 1 \pmod e$ have been investigated by many authors (see, \cite{KR1985} and the references 
therein). Cyclotomic numbers of small composite order have been investigated by many authors, but most of the results involve the classical ambiguity. 
The problem of removal of this ambiguity may led to composite moduli. The first result in composite case is due to L. E. Dickson \cite{DI1935} who had 
computed cyclotomic numbers of orders $4, 6, 8, 10$ and $12$ over the field $\mathbb{F}_q$ for $r=1$. 
In \cite{ST1967}, the author showed that the cyclotomic numbers of orders $3, 4, 6$ over the field $\mathbb{F}_q$ are well determined by a unique 
representation of $p^r$ in terms of a particular binary quadratic form $x^2+dy^2$, and that of order $8$ are determined by two such forms.         
L. E. Dickson also determined the cyclotomic numbers of orders $ 14, 22$ in \cite{DI1935-2}, and that of orders $9, 15, 16, 18, 20, 24$ in \cite{DI1935-3} 
over the field $\mathbb{F}_q$ for $r=1$. Following the techniques of \cite{DI1935}, A. L. Whiteman \cite{WH1960} computed the cyclotomic numbers of order 
$12$ over the field $\mathbb{F}_q$ for $r=1$. He also computed the cyclotomic numbers of order $16$ over the same field in \cite{WH1957}. These cyclotomic 
numbers of order $16$ were also determined independently by E. Lehmer in \cite{LE1954} and  R. J. Evans and J. R. Hill in \cite{EH1979}. 
J. B. Muskat and A. L. Whiteman gave the complete theory for the cyclotomic numbers of order $20$ in \cite{MW1970}. J. B. Muskat analysed completely 
the cyclotomic numbers of order $14$ in \cite{MU1966} and that of orders $24, 30$ in \cite{MU1968}. In \cite{FMSW1986}, the authors discussed the 
cyclotomic numbers of order $15$ over the field $\mathbb{F}_{q}$ for $q=p^2$ when $p\equiv 4, 11\pmod{15}$. 
For $\ell$ an odd prime, Katre and Rajwade \cite{KR1985} solved the cyclotomy problem of order $\ell$ whereas Acharya and Katre \cite{AK1995} determined 
the cyclotomic numbers of order $2\ell$ over the field $\mathbb{F}_q$ for $q=p^r$ with the prime $p\equiv 1\pmod \ell$ in terms of the solutions of certain 
Diophantine systems.
Recently  Shirolkar and  Katre \cite{SK2011} determined the cyclotomic numbers of order $\ell^2$ with prime $\ell$ in terms of the coefficients of 
Jacobi sums of orders $\ell$ and $\ell^2$. 

In this paper, We obtain a formula for the cyclotomic numbers of order $2\ell^2$ over the finite field $\mathbb{F}_q$ for 
$q=p^r\equiv 1\pmod {2\ell^2}$ in terms of the coefficients of Jacobi sums of orders $\ell$, $2\ell$, $\ell^2$ and $2\ell^2$. These coefficients can be 
evaluated in terms of the Dickson-Hurwitz sums of of orders $\ell$, $2\ell$, $\ell^2$ and $2\ell^2$. We also discuss the nature of two types of cyclotomic 
numbers in terms of their associated matrix representations.

\section{Preliminaries}
For an odd prime $\ell$, let $p$ be a prime and $q=p^r$ with an integer $r\geq 1$ satisfying $q\equiv 1\pmod {2\ell^2}$. We write $q=2\ell^2k+1$ for 
some positive integer $k$. In next two subsequent subsection we state some well known properties \cite{SK2011} of cyclotomic numbers and Dickson-Hurwitz sums which are 
needed to give an explicit expression for cyclotomic numbers of order $2\ell^{2}$
\subsection{Properties of Cyclotomic Numbers}
\begin{enumerate}
 \item  $(a,b)_{2\ell^2}=(a' ,b')_{2\ell^2}$ whenever $a\equiv a'\pmod{2\ell^2}$ and $b\equiv b'\pmod{2\ell^2}$.\\
 \item $(a,b)_{2\ell^2}= (2\ell^2-a,b-a)_{2\ell^{2}}$ along with the following:
\begin{equation*}\label{2.1}
(a,b)_{2\ell^{2}}=\begin{cases}
 (b,a)_{2\ell^{2}}\hspace*{1.751cm} \text{ if }  k \text{ is  even  or  } q=2^r,\\
(b+\ell^{2},a+\ell^{2})_{2\ell^{2}}\hspace*{2mm} \text{ if } k \text{ is odd}.
\end{cases}. 
\end{equation*}
\item \begin{equation} \label{2.2}
\sum_{a=0}^{2\ell^{2}-1}\sum_{b=0}^{2\ell^{2}-1}(a,b)_{2\ell^{2}}=q-2.
\end{equation}
\item \begin{equation*} \label{2.3}
\sum_{b=0}^{2\ell^{2}-1}(a,b)_{2\ell^{2}}=k-n_{a},
\end{equation*}
where $n_{a}$ is given by 
\begin{equation*}
n_{a}=\begin{cases}
1 \quad \text{ if } a=0, 2\mid k \text{ or if } a=\ell^2, 2\nmid k;\\
0 \quad \text{otherwise}.
\end{cases}
\end{equation*}
\item \begin{equation} \label{2.4}
\sum_{a=0}^{2\ell^{2}-1}(a,b)_{2\ell^{2}}=\begin{cases}
k-1 \hspace*{3mm}\text{ if } b =0; \\ 
k \hspace*{1cm}\text{ if } 1\leq b \leq 2\ell^{2}-1.
\end{cases}
\end{equation}
\end{enumerate}

\subsection{Properties of Dickson-Hurwitz sums}
Dickson-Hurwitz sums, $B_{2\ell^2}(a,b)$ \cite{ZE1971} of order $2\ell^2$ are defined for two positive integers $a$ and $ b$ modulo $2\ell^2$ by
\begin{equation} \label{2.5}
B_{2l^{2}}(a,b)=\sum_{h=0}^{2\ell^{2}-1}(h,a-bh)_{2\ell^{2}}.
\end{equation}
\begin{enumerate}
 \item \begin{equation}\label{2.6}
\sum_{a=0}^{2l^{2}-1}B_{2l^{2}}(a,b)=q-2.
\end{equation}
\item \begin{equation}\label{2.7}
B_{2l^{2}}(a,0)=\begin{cases}
k-1 \hspace*{3mm}\text{ if } a =0; \\ 
k \hspace*{1cm}\text{ if } 1\leq a \leq 2\ell^{2}-1.
\end{cases}
\end{equation}
\item \begin{align}\label{xx}
B_{2\ell^{2}}(\ell^{2}+k,n) &\equiv \sum_{a=0}^{2\ell^{2}-1}B_{2\ell^{2}}(a-\ell,n)\pmod{2\ell^{2}}\nonumber \\
&\equiv \sum_{a=0}^{2\ell^{2}-1}B_{2\ell^{2}}(a,n)\pmod{2\ell^{2}}\nonumber\\
&\equiv q-2\pmod {2\ell^2},
\end{align}
where $0\leq k \leq \phi(2\ell^{2})-1$.
\end{enumerate}
\subsection{Relations between Jacobi Sums and Dickson-Hurwitz Sums}
For any positive integers $m$ and $n$ modulo $2\ell^2$, we recall the following relation (see, relation (8) in {\cite{DI1935-2}}),
\begin{equation} \label{2.8}
J_{2l^{2}}(m,n)=\sum_{a=0}^{2\ell^{2}-1}\zeta_{2\ell^2}^{na}\sum_{b=0}^{2\ell^{2}-1}\zeta_{2\ell^2}^{-(m+n)b}(a,b)_{2\ell^{2}},
\end{equation}
where $\zeta_{2\ell^2}$ is a primitive $2\ell^{2}$-th root of unity.

We also recall (see, relation (9) in {\cite{DI1935-2}}),
\begin{equation} \label{2.9}
J_{2\ell^{2}}(m,n)=(-1)^{nk}J_{2\ell^{2}}(-m-n,n).
\end{equation}
Puttig $m=2$ in \eqref{2.9} and $m=-2-n$ in \eqref{2.8}, we obtain respectively:
\begin{equation*}
J_{2\ell^{2}}(2,n)=(-1)^{nk}J_{2\ell^{2}}(-2-n,n)
\end{equation*}
and
\begin{equation*}
J_{2\ell^{2}}(-2-n,n)=\sum_{a=0}^{2\ell^{2}-1}\sum_{b=0}^{2\ell^{2}-1}\zeta_{2\ell^2}^{na+2b}(a,b)_{2\ell^{2}}.
\end{equation*}

Eliminating $b$ by the use of $na+2b\equiv i \pmod {\ell^{2})}$, we obtain
\begin{equation*}
J_{2\ell^{2}}(-2-n,n)=\sum_{i=0}^{\ell^{2}-1}\sum_{a=0}^{\ell^{2}-1}\zeta_{\ell^2}^{i}\big(a,\frac{i-na}{2}\big)=(-1)^{nk} J_{2\ell^{2}}(2,n),
\end{equation*}
Writing $B_{\ell^{2}}(i,n)=\sum_{a=0}^{\ell^{2}-1}(a,(i-na)/2)$, we get

\begin{equation} \label{2.10}
(-1)^{nk}J_{2\ell^{2}}(2,n)=\sum_{i=0}^{\ell^{2}-1}B_{\ell^{2}}(i,n)\zeta_{\ell^2}^{i}.
\end{equation}

where $\zeta_{\ell^2}$ is a $\ell^{2}$-th root of unity.\\
Also from \cite{DI1935-2}, we have
\begin{equation}\label{2.11}
(-1)^{nk}J_{2\ell^{2}}(1,n)=\sum_{i=0}^{2\ell^{2}-1}B_{2\ell^{2}}(i,n)\zeta_{2\ell^2}^{i}.
\end{equation}
\section{Some results on Dickson-Hurwitz sums}
In this section, we determine the trace of $J_{2\ell^2}(1,n)\zeta_{2\ell^2}^{-t}$, $J_{2\ell^2}(2, n)\zeta_{2\ell^2}^{-t}$ and also obtain a result that 
shows a relation between the coefficients $d_{i,n}$ and Dickson-Hurwitz sums of order $2\ell^{2}$ which are needed in the subsequent sections.  
Let $\zeta_\ell$ be a primitive $\ell$-th root of unity. Then 
$Tr_{\mathbb{Q}(\zeta_{2\ell^2})/\mathbb{Q}}(\zeta_\ell)=-\ell$,  $ Tr_{\mathbb{Q}(\zeta_{2\ell^2})/\mathbb{Q}}(\zeta_{\ell^2})=\ell$ and 
$Tr_{\mathbb{Q}(_{2\ell^2})/\mathbb{Q}}(\zeta_{2\ell^2})=0$.
Then using \eqref{2.11} we compute for a positive integer $t$,
\begin{align}\label{2.12}
&Tr_{\mathbb{Q}(\zeta_{2\ell^2})/\mathbb{Q}}(J_{2\ell^{2}}(1,n)\zeta_{2\ell^2}^{-t})
=\sum_{i=0}^{2\ell^{2}-1}B_{2\ell^2}(i,n)Tr_{\mathbb{Q}(\zeta_{2\ell^2})/\mathbb{Q}}(\zeta_{2\ell^2}^{i-t})\nonumber\\
&= \ell(\ell-1)B_{2\ell^2}(t,n)+\sum_{\substack{i=0\\ i\neq t}}^{2\ell^{2}-1}B_{2\ell^2}(i,n)Tr_{\mathbb{Q}(\zeta_{2\ell^2})/\mathbb{Q}}(\zeta_{2\ell^2}^{i-t})\nonumber\\
& = \ell(\ell-1)B_{2\ell^2}(t,n)-\ell\sum_{u=1}^{2l-1}B_{2\ell^2}(\ell u+t,n)+\ell\sum_{u=1}^{\ell-1}B_{2\ell^2}(2\ell u+t,n).
\end{align}
Similary using \eqref{2.10}, we obtain
\begin{align}\label{2.13}
Tr_{\mathbb{Q}(\zeta_{\ell^2})/\mathbb{Q}}(J_{2\ell^{2}}(2,n)\zeta_{\ell^2}^{-t})
&=\sum_{i=0}^{\ell^{2}-1}B_{\ell^2}(i,n)Tr_{\mathbb{Q}(\zeta_{\ell^2})/\mathbb{Q}}(\zeta_{\ell^2}^{i-t})\nonumber\\
&=\ell(\ell-1)B_{\ell^2}(t,n)-l\sum_{u=1}^{l-1}B_{\ell^2}(lu+t,n).
\end{align}
Recently, we \cite{AT2018} determined the coefficients of Jacobi sums of order $2\ell^{2}$ in terms of Dickson-Hurwitz sums of the same order. More 
precisely, we proved:
\begin{theorem} \label{AT1}
Let $\ell$ be an odd prime. Let $p$ be a prime such that $q=p^{r}\equiv 1 \pmod {2\ell^{2}}$ with an integer $r\geq 1$. Then for any positive integer $n$,
$$J_{2\ell^{2}}(1,n)=\sum_{i=0}^{\ell(\ell-1)-1} d_{i,n}\zeta_{2\ell^2}^{i},$$
where the coefficents $d_{i, n}$ are given by 
$$
d_{i,n}=B_{2\ell^{2}}(i,n)\mp B_{2\ell^{2}}(\ell(\ell-1)+j,n)-B_{2\ell^{2}}(\ell^{2}+k,n)\pm B_{2\ell^{2}}((2\phi(\ell^{2})+\ell)+j,n),
 $$
satisfying $ 0\leq j \leq \ell-1$, $j\equiv i \pmod { 2\ell^{2}}$, $k\equiv i \pmod{2\ell^{2}}$, $0\leq i, k \leq \phi(2\ell^{2})-1$.
\end{theorem}
The next result gives a relation between the coeffients $d_{i,n}$ and Dickson-Hurwitz sums of order $2\ell^2$. This result is useful to determine the cyclotomic numbers of order $2\ell^2$. Here, we prove:
\begin{theorem}\label{thm1}
Let $\ell$ be an odd prime, and $t$ and $n$ be two positive integers modulo $2\ell^{2}$. Then 
\begin{eqnarray*}
& \ell(\ell-1)B_{2\ell^2}(t,n)-\ell\sum\limits_{u=1}^{2\ell-1}B_{2\ell^2}(\ell  u+t,n)+\ell\sum\limits_{u=1}^{\ell-1}B_{2\ell^2}(2\ell u+t,n)\\
&= \ell(\ell-1)d_{t,n}-\ell\sum\limits_{u=\ell-1}^{2\ell-3}d_{u\ell+t,n}-\ell(q-2),
\end{eqnarray*}
where $t=j\ell+s$ with $0 \leq j \leq 2\ell-1$ and $0 \leq s \leq \ell-1$. 
\end{theorem}

\begin{proof}
Let $$D(n)=\ell(\ell-1)B_{2\ell^2}(t,n)-\ell\sum\limits_{u=1}^{2\ell-1}B_{2\ell^2}(\ell  u+t,n)+\ell\sum\limits_{u=1}^{\ell-1}B_{2\ell^2}(2\ell u+t,n).$$ Then putting $t=j\ell+s$ with $0 \leq j \leq 2\ell-1$ and $0 \leq s \leq \ell-1$, we obtain
 $$D(n)=\ell(\ell-1)B_{2\ell^2}(j\ell+s,n)-\ell\sum\limits_{u=1}^{2\ell-1}B_{2\ell^2}(\ell(u+j)+s,n)+\ell\sum\limits_{u=1}^{\ell-1}B_{2\ell^2}(\ell(2u+j)+s,n).$$
 This can be written as
  \begin{align*}
  D(n)&=\ell(\ell -1)B_{2\ell^2}(j\ell +s,n)-\ell \sum\limits_{u=1}^{2\ell -j-3}B_{2\ell^2}(\ell (u+j)+s,n) \\
  &- \ell \sum\limits_{u=2\ell -j-1}^{2\ell -1}B_{2\ell^2}(\ell (u+j)+s,n)-\ell B_{2\ell^2}((\ell (2\ell-j-2+j))+s,n)\\
  & +\ell \sum\limits_{u=1}^{\ell-1}B_{2\ell^2}(\ell (2u+j)+s,n).
 \end{align*}
 Since $\ell B_{2\ell^2}((\ell (2\ell-j-2+j))+s,n)=B_{2\ell^2}(\ell^{2}+k,n))$, so that
  \begin{align*}
  D(n)&=\ell(\ell-1)(B_{2\ell^2}(j\ell+s,n)-B_{2\ell^2}(\ell^{2}+k,n))+\ell\sum\limits_{u=1}^{\ell-1}B_{2\ell^2}(\ell(2u+j)+s,n)\\ 
  & -\ell \bigg(\sum\limits_{u=1}^{2\ell-j-3}B_{2\ell^2}(\ell(u+j)+s,n)-B_{2\ell^2}(\ell^{2}+k,n)\bigg)-\ell \bigg(\sum\limits_{u=2\ell-j-1}^{2\ell-1}B_{2\ell^2}(\ell(u+j)\\ 
  &  +s,n)-B_{2\ell^2}(\ell^{2}+k,n)\bigg)-\ell^{2}B_{2\ell^2}(\ell^{2}+k,n).
 \end{align*}
 Applying \eqref{xx}, we obtain
\begin{align*}
D(n)&=\ell(\ell-1)(B_{2\ell^2}(j\ell+s,n)\mp B_{2\ell^2}(\ell(\ell-1)+s,n)\pm B_{2\ell^2}(2\ell^{2}-\ell+s,n)\\ 
& -B_{2\ell^2}(\ell^{2}+k,n)) +l\sum\limits_{u=1}^{\ell-1}(B_{2\ell^2}(\ell(2u+j)+s,n)\mp B_{2\ell^2}(\ell(\ell-1)+s,n)\\ 
& \pm B_{2\ell^2}(2\ell^{2}-\ell+s,n)-B_{2\ell^2}(\ell^{2}+k,n)) -\ell \bigg(\sum\limits_{u=1}^{2\ell-j-3}B_{2\ell^2}(\ell(u+j)+s,n)\\ 
& \mp B_{2\ell^2}(\ell(\ell-1)+s,n)\pm B_{2\ell^2}(2\ell^{2}-\ell+s,n)-B_{2\ell^2}(\ell^{2}+k,n)\bigg)\\ 
&-\ell \bigg(\sum\limits_{u=2\ell-j-1}^{2\ell-1}B_{2\ell^2}(\ell(u+j)+s,n)\mp B_{2\ell^2}(\ell(\ell-1)+s,n)\pm B_{2\ell^2}(2\ell^{2}-\ell+s,n)\\ 
&-B_{2\ell^2}(\ell^{2}+k,n)\bigg)-\ell(q-2).
\end{align*}
We now applying Theorem \ref{AT1} and \eqref{xx} to get,
\begin{align*}
D(n)&=\ell(\ell-1)d_{j\ell+s,n}+\ell\sum_{u=1}^{\ell-1}d_{\ell(2u+j)+s,n}-\ell
\sum_{u=1}^{2\ell-j-3}d_{\ell(u+j)+s,n}\\
& -\ell\sum_{u=2\ell-j-1}^{2\ell-1}d_{\ell (u+j)+s,n}-\ell(q-2). 
\end{align*}

Puting $u+j=x$ in second sum and $u+j\equiv x \ \pmod{ (2\ell-1)}$ in third sum, we have
$$
D(n)=\ell(\ell-1)d_{j\ell+s,n}+\ell\sum_{u=1}^{\ell-1}d_{\ell(2u+j)+s,n}-\ell
\sum_{x=0}^{2\ell-3}d_{\ell x+s,n}-\ell(q-2).
$$

Consider $u\ell+t\equiv \ell x+s \pmod{ \ell(\ell-1)}$. Then 
$$
D(n)=\ell(\ell-1)d_{t,n}+\ell\sum_{u=1}^{\ell-1}d_{\ell(2u+j)+s,n}-\ell\sum_{u=0}^{2l-3}d_{ul+t,n}-\ell(q-2).
$$
Puting $2u+j\equiv m+j \pmod{(2\ell-1)}$ to obtain
$$
D(n)=\ell(\ell-1)d_{t,n}+\ell\sum_{u=0}^{\ell-2}d_{u\ell+t,n}-\ell\sum_{u=0}^{2\ell-3}d_{ul+t,n}-\ell(q-2).
$$
Since for every $m+j, u\ell+t\equiv (m+j)\ell+s \pmod{\ell(\ell-1)}$ with $u\in \{0,1,2,3,\cdots, \ell-2\}$, we obtain
$$
D(n)=\ell (\ell -1)d_{t,n}+\ell \sum_{u=0}^{\ell-2}d_{u\ell +t,n}-\ell \sum_{u=0}^{2\ell -3}d_{u\ell +t,n}-
\ell (q-2).$$
This further implies 
$$D(n)=\ell(\ell-1)d_{t,n}-\ell\sum_{u=\ell-1}^{2\ell-3}d_{u\ell+t,n}-\ell(q-2).
$$
\end{proof}
\section{Main Results}
In this section, we obtain an explicit expression for cyclotomic numbers of order $2\ell^{2}$ in terms of the coefficients of Jacobi sums of orders $\ell$, $2\ell$, $\ell^{2}$ and $2\ell^{2}$. These coefficients of Jacobi sums can be evaluated in terms of Dickson-Hurwitz sums of orders $\ell$, $2\ell$, $\ell^{2}$ and $2\ell^{2}$. Here, we mainly prove:
\begin{theorem} \label{thm2}
Let $\ell$ be an odd prime. Let $p$ be a prime such that $q=p^r\equiv 1 \pmod {2\ell^{2}}$ with an integer $r\geq 1$. Then 
\begin{align*}
&4\ell^4 (a,b)_{2l^{2}}= \ell^{4}(a,b)_{\ell^{2}}+4\ell^{2}(a,b)_{2\ell}-\ell^{2}(a,b)_{\ell}-\ell(q-2)(4\ell^{2}-3)+\sum_{i=2}^{2\ell^{2}-1}\{\ell(\ell-1)\\
&\times d_{ia+b,n}-\ell\sum_{u=\ell-1}^{2\ell-3}d_{u\ell+ia+b,n}\} +\sum_{j=1}^{2\ell^{2}-1}\{\ell(\ell-1)d_{a+jb,n}-\ell\sum_{u=\ell-1}^{2\ell-3}
d_{u\ell+a+jb,n}\}\\
&-\sum_{i\text{ odd}}^{2\ell-1}\{\varepsilon((\ell^{2}+1)/2)(\ell ia+2b))b_{((\ell^{2}+1)/2)(\ell ia+2b),n} -\ell\sum_{u=0}^{\ell-2}b_{u\ell+((\ell^{2}+1)/2)(\ell ia+2b),n}\}\\
& -\sum_{j\text{ odd}}^{2\ell-1}\{\varepsilon((\ell^{2}+1)/2)(2a+\ell jb))b_{((\ell^{2}+1)/2)(2a+\ell jb),n} -\ell\sum_{u=0}^{\ell-2}b_{u\ell+((\ell^{2}+1)/2)(2a+\ell jb),n}\},
\end{align*}
where $\varepsilon(t)$ is given by
\begin{equation}\label{3.1}
\varepsilon(t)=
\begin{cases}
\ell^2\hspace*{4mm} \text{if } 0\leq j\leq \ell-2, \text{ i.e. } 0\leq t\leq \ell^2-\ell,\\
-\ell \hspace*{3mm} \text{if } j=\ell-1, \text{i.e. }  \ell^2-\ell\leq t\leq \ell^2-1.
\end{cases}
\end{equation}
\end{theorem}
To prove this theorem, we need an important result of D. Shirolkar and S. A. Katre \cite[Lemma 6.1]{SK2011}. We recall that results for convenience. We consider the Jacobi sum $J_{\ell^2}(1,n)$ of order $\ell^2$, $$J_{\ell^2}(1,n)=\sum_{i=0}^{\ell(\ell-1)}b_{i,n}\zeta_{\ell^2}^i.$$

\begin{lemma}\label{lm1}
For two positive integers $t$ and $n$ modulo $\ell^2$, define
$$C(t,n):= \ell(\ell-1)B_{\ell^2}(t,n)-\ell B_{\ell^2}(u\ell+t, n).$$
Let $0\leq t\leq \ell^2-1$. Write $t=j\ell+s$, where $0\leq j\leq \ell^2-1$ and $0\leq s\leq \ell^2-1$. Then 
$$C(t,n)=\varepsilon(t)b_{t,n}-\ell\sum_{\ell=0}^{\ell-2}b_{u\ell+t, n},$$
where $\varepsilon(t)$ is given by \eqref{3.1}.
\end{lemma}

\subsection*{Proof of Theorem \ref{thm2}}
Applying \eqref{00}, we obtain
$$
4\ell^{4}(a,b)_{2\ell^{2}} = \sum_{i}\sum_{j}J_{2\ell^{2}}(i,j) \zeta_{2\ell^2}^{-(ai+bj)}.
$$
Let us consider $G=\text{Gal}(\mathbb{Q}(\zeta_{2\ell^2})/\mathbb{Q})$ and $G'=\text{Gal}(\mathbb{Q}(\zeta_{\ell^2})/\mathbb{Q})$. Then the above equation can be written as 

\begin{align*}
4\ell^{4}(a,b)_{2\ell^{2}} & =(a,b)_{\ell^{2}}+\sum_{i \text{ odd }}^{2\ell^{2}-1}J_{2\ell^{2}}(i,0)\zeta_{2\ell^2}^{-ia}+\sum_{j \text{ odd}}^{2\ell^{2}-1}J_{2\ell^{2}}(0,j)\zeta_{2\ell^2}^{-jb} \\ &+\sum_{i=1}^{2\ell^{2}-1}\sum_{\sigma \in G}\sigma (J_{2\ell^{2}}(i,1)\zeta_{2\ell^2}^{-(ia+b)})+\sum_{j=1}^{2\ell^{2}-1}\sum_{\sigma \in G}\sigma (J_{2\ell^{2}}(1,j)\zeta_{2\ell^2}^{-(a+jb)}) \\ &-\sum_{\sigma \in G}\sigma (J_{2\ell^{2}}(1,1)\zeta_{2\ell^2}^{-(a+b)})+\sum_{i\text{ odd}}^{2\ell-1}\sum_{\sigma \in G'}\sigma (J_{2\ell^{2}}(\ell i,2)\zeta_{2\ell^2}^{-(\ell ia+2b)})\\ &+\sum_{j\text{ odd}}^{2\ell-1}\sum_{\sigma \in G'}\sigma (J_{2\ell^{2}}(2,\ell j)\zeta_{2\ell^2}^{-(2a+\ell jb)})+\sum_{i=1}^{2\ell-1}\sum_{j=1}^{2\ell-1}J_{2\ell^{2}}(\ell i,\ell j)\zeta_{2\ell^2}^{-(\ell ia+\ell jb)}\\ &-\sum_{i=1}^{\ell-1}\sum_{j=1}^{\ell-1}J_{2\ell^{2}}(2\ell i,2\ell j)\zeta_{2\ell^2}^{-(2\ell ia+2\ell jb)}.
\end{align*}
Since the second and third terms in R H S are zero, so that 
\begin{align*}
4\ell^{4}(a,b)_{2\ell^{2}} & = \ell^{4}(a,b)_{\ell^{2}}+\sum_{i=1}^{2\ell^{2}-1}\sum_{\sigma \in G}\sigma (J_{2\ell^{2}}(i,1)\zeta_{2\ell^2}^{-(ia+b)})+\sum_{j=1}^{2\ell^{2}-1}\sum_{\sigma \in G}\sigma (J_{2\ell^{2}}(1,j)\zeta_{2\ell^2}^{-(a+jb)})\\
&-\sum_{\sigma \in G}\sigma (J_{2\ell^{2}}(1,1)\zeta_{2\ell^2}^{-(a+b)})-\sum_{i\text{ odd}}^{2\ell-1}\sum_{\sigma \in G'}\sigma (J_{2\ell^{2}}(\ell i,2)\zeta_{\ell^2}^{-((\ell^{2}+1)/2)(\ell ia+2b)})\\ & -\sum_{j\text{ odd}}^{2\ell-1}\sum_{\sigma \in G'}\sigma (J_{2\ell^{2}}(2,\ell j)\zeta_{\ell^2}^{-((\ell^{2}+1)/2)(2a+\ell jb)}) +\sum_{i=1}^{2\ell-1}\sum_{j=1}^{2\ell-1}J_{2\ell^{2}}(\ell i,\ell j)\zeta_{2\ell^2}^{-(\ell ia+\ell jb)}\\ & -\sum_{i=1}^{\ell-1}\sum_{j=1}^{\ell-1}J_{2\ell^{2}}(2\ell i,2\ell j)\zeta_{2\ell^2}^{-(2\ell ia+2\ell jb)}.
\end{align*}
This implies,
\begin{align*}
4\ell^{4}(a,b)_{2\ell^{2}} & = \ell^{4}(a,b)_{\ell^{2}}+\sum_{i=1}^{2\ell^{2}-1}Tr_{\mathbb{Q}(\zeta_{2\ell^2})/\mathbb{Q}}(J_{2\ell^{2}}(i,1)\zeta_{2\ell^2}^{-(ia+b)})
 \\ & +\sum_{j=1}^{2\ell^{2}-1}Tr_{\mathbb{Q}(\zeta_{2\ell^2})/\mathbb{Q}}(J_{2\ell^{2}}(1,j)\zeta_{2\ell^2}^{-(a+jb)})-Tr_{\mathbb{Q}(\zeta_{2\ell^2})/\mathbb{Q}}(J_{2\ell^{2}}(1,1)\zeta_{2\ell^2}^{-(a+b)}) \\ & -\sum_{i\text{ odd}}^{2\ell-1}Tr_{\mathbb{Q}(\zeta_{\ell^2})/\mathbb{Q}}(J_{2\ell^{2}}(\ell i,2)\zeta_{\ell^2}^{-((\ell^{2}+1)/2)(\ell ia+2b)})\\ &- \sum_{j\text{ odd}}^{2\ell-1}Tr_{\mathbb{Q}(\zeta_{\ell^2})/\mathbb{Q}}(J_{2\ell^{2}}(2,\ell j)\zeta_{\ell^2}^{-((\ell^{2}+1)/2)(2a+\ell jb)}) \\ & +\sum_{i=1}^{2\ell-1}\sum_{j=1}^{2\ell-1}J_{2\ell^{2}}(\ell i,\ell j)\zeta_{2\ell^2}^{-(\ell ia+\ell jb)} -\sum_{i=1}^{\ell-1}\sum_{j=1}^{\ell-1}J_{2\ell^{2}}(2\ell i,2\ell j)\zeta_{2\ell^2}^{-(2\ell ia+2\ell jb)}.
\end{align*}
Applying \eqref{00}, we get
\begin{align*}
4\ell^{4}(a,b)_{2\ell^{2}} & = \ell^{4}(a,b)_{\ell^{2}}+\sum_{i=1}^{2\ell^{2}-1}Tr_{\mathbb{Q}(\zeta_{2\ell^2})/\mathbb{Q}}(J_{2\ell^{2}}(i,1)\zeta_{2\ell^2}^{-(ia+b)})\nonumber
 \\ & +\sum_{j=1}^{2\ell^{2}-1}Tr_{\mathbb{Q}(\zeta_{2\ell^2})/\mathbb{Q}}(J_{2\ell^{2}}(1,j)\zeta_{2\ell^2}^{-(a+jb)})-Tr_{\mathbb{Q}(\zeta_{2\ell^2})/\mathbb{Q}}(J_{2\ell^{2}}(1,1)\zeta_{2\ell^2}^{-(a+b)})\nonumber \\ & -\sum_{i\text{ odd}}^{2\ell-1}Tr_{\mathbb{Q}(\zeta_{\ell^2})/\mathbb{Q}}(J_{2\ell^{2}}(\ell i,2)\zeta_{\ell^2}^{-((\ell^{2}+1)/2)(\ell ia+2b)})\nonumber \\ &- \sum_{j\text{ odd}}^{2\ell-1}Tr_{\mathbb{Q}(\zeta_{\ell^2})/\mathbb{Q}}(J_{2\ell^{2}}(2,\ell j)\zeta_{\ell^2}^{-((\ell^{2}+1)/2)(2a+\ell jb)})\nonumber \\ & + 
4\ell^2(a, b)_{2\ell}-\ell^2(a,b)_{\ell}.
\end{align*}
By the use of \eqref{2.12} and \eqref{2.13}, we get
\begin{align*}
4\ell^{4}(a,b)_{2\ell^{2}} & = \ell^{4}(a,b)_{\ell^{2}}+\sum_{i=1}^{2\ell^{2}-1}\{\ell(\ell-1)B_{2\ell^2}(ia+b,n)-\ell\sum_{u=1}^{2\ell-1}B_{2\ell^2}(\ell u+ia+b,n)\\ & +\ell\sum_{u=1}^{\ell-1}B_{2\ell^2}(2\ell u+ia+b,n)\}  +\sum_{j=1}^{2\ell^{2}-1}\{\ell(\ell-1)B_{2\ell^2}(a+jb,n)\\ & -\ell\sum_{u=1}^{2\ell-1}B_{2\ell^2}(\ell u+a+jb,n)+\ell\sum_{u=1}^{\ell-1}B_{2\ell^2}(2\ell u+a+jb,n)\}  \\ & - \ell(\ell-1)B_{2\ell^2}(a+b,n)+\ell\sum_{u=1}^{2\ell-1}B_{2\ell^2}(\ell u+a+b,n)\\ & -\ell\sum_{u=1}^{\ell-1}B_{2\ell^2}(2\ell u+a+b,n) -\sum_{i\text{ odd}}^{2\ell-1}\{\ell(\ell-1)B_{\ell^2}(((\ell^{2}+1)/2)(\ell ia+2b),n)\\ & -\ell\sum_{u=1}^{\ell-1}B_{\ell^2}(u\ell+((\ell^{2}+1)/2)(\ell ia+2b),n)\}  -\sum_{j\text{ odd}}^{2\ell-1}\{\ell(\ell-1)
 B_{\ell^2}(((\ell^{2}+1)/2)\\
 &\times (2a+\ell jb),n)-\ell\sum_{u=1}^{\ell-1}B_{\ell^2}(u\ell+((\ell^{2}+1)/2)(2a+\ell jb),n)\}  +4\ell^{2}(a,b)_{2\ell}\\ &
 -\ell^{2}(a,b)_{\ell}. 
\end{align*}
Employing Theorem \ref{thm1} and Lemma \ref{lm1}, we obtain
\begin{align*}
& 4\ell^{4}(a,b)_{2\ell^{2}}  = \ell^{4}(a,b)_{\ell^{2}}+4\ell^{2}(a,b)_{2\ell}-\ell^{2}(a,b)_{\ell}+\sum_{i=2}^{2\ell^{2}-1}\{\ell(\ell-1)d_{ia+b,n}-\ell(q-2)\\
& 
-\ell\sum_{u=\ell-1}^{2\ell-3}d_{u\ell+ia+b,n}\} +\sum_{j=1}^{2\ell^{2}-1}\{\ell(\ell-1)d_{a+jb,n} -\ell(q-2)-\ell\sum_{u=\ell-1}^{2\ell-3}d_{u\ell+a+jb,n}\}\\
& -\sum_{i\text{ odd}}^{2\ell-1}\{\varepsilon(((\ell^{2}+1)/2)(\ell ia+2b))b_{((\ell^{2}+1)/2)(\ell ia+2b),n} -\ell \sum_{u=0}^{\ell-2}b_{u\ell+((\ell^{2}+1)/2)(\ell ia+2b),n}\} \\ &
-\sum_{j\text{ odd}}^{2\ell-1}\{\varepsilon(((\ell^{2}+1)/2)(2a+\ell jb))b_{((\ell^{2}+1)/2)(2a+\ell jb),n} -\ell\sum_{u=0}^{\ell-2}b_{u\ell+((\ell^{2}+1)/2)(2a+\ell jb),n}\}.
\end{align*}  
This can be further simplified as
\begin{align*}
4\ell^{4}(a,b)_{2\ell^{2}} & = \ell^{4}(a,b)_{\ell^{2}}+4\ell^{2}(a,b)_{2\ell}-\ell^{2}(a,b)_{\ell}-\ell(q-2)(4\ell^{2}-3)\\
& +\sum_{i=2}^{2\ell^{2}-1}\{\ell(\ell-1)d_{ia+b,n}-\ell\sum_{u=\ell-1}^{2\ell-3}
d_{u\ell+ia+b,n}\} +\sum_{j=1}^{2\ell^{2}-1}\{\ell(\ell-1)d_{a+jb,n}\\ & -\ell\sum_{u=\ell-1}^{2\ell-3}d_{u\ell+a+jb,n}\}-\sum_{i\text{ odd}}^{2\ell-1}\{\varepsilon(((\ell^{2}+1)/2)(\ell ia+2b))b_{((\ell^{2}+1)/2)(\ell ia+2b),n}\\ & -\ell\sum_{u=0}^{\ell-2}b_{u\ell+((\ell^{2}+1)/2)(\ell ia+2b),n}\}-\sum_{j\text{ odd}}^{2\ell-1}\{\varepsilon(((\ell^{2}+1)/2)(2a+\ell jb))\\ & \times b_{((\ell^{2}+1)/2)(2a+\ell jb),n} -\ell\sum_{u=0}^{\ell-2}b_{u\ell+((\ell^{2}+1)/2)(2a+\ell jb),n}\}.
\end{align*}
This completes the proof.

Substituting the well known formulae for cyclotomic numbers of order $\ell$ and $2\ell$ in Theorem \ref{thm2}, we obtain the following: 
\begin{corollary} Let $p, q$ and $\ell$ as in Theorem \ref{thm2} with $r=1$. Then \begin{align*}
&4\ell^{4}(a,b)_{2\ell^2} = \ell^{4}(a,b)_{\ell^{2}}-\ell(q-2)(4\ell^{2}-3)-\{(-1)^{b}+(-1)^{a+k}+(-1)^{a+b}\}\\
& \times \{\ell+\sum_{m=1}^{\ell-1}b_{m}(\ell) +\sum_{u=0}^{\ell-2}\sum_{m=1}^{\ell-1}b_{m}(2u+1)\}+(-1)^{b}\ell\{b_{\nu(-a)}(\ell)\\
& +\sum_{u=0}^{\ell-1}b_{\nu(b-2au-2a)}(2u+1)\}+(-1)^{a+b}\ell\{b_{\nu(b)}(\ell)  +\sum_{u=0}^{\ell-1}b_{\nu(a+2bu+b)}(2u+1)\}\\
& +(-1)^{a+k}\ell\{b_{\nu(-b)}(\ell) +\sum_{u=0}^{\ell-1}b_{\nu(a-2bu-2b)}(2u+1)\}+\sum_{i=2}^{2\ell^{2}-1}\{\ell(\ell-1)d_{ia+b,n}\\ & -\ell\sum_{u=\ell-1}^{2\ell-3}d_{u\ell+ia+b,n}\} +\sum_{j=1}^{2\ell^{2}-1}\{\ell(\ell-1)d_{a+jb,n}-\ell\sum_{u=\ell-1}^{2\ell-3}d_{u\ell+a+jb,n}\}\\ & -\sum_{i\text{ odd }}^{2\ell-1}\{\varepsilon(((\ell^{2}+1)/2)(\ell ia+2b))b_{((\ell^{2}+1)/2)(\ell ia+2b),n} -\ell \sum_{u=0}^{\ell-2}b_{u\ell+((\ell^{2}+1)/2)(\ell ia+2b),n}\}\\ & -\sum_{j\text{ odd}}^{2\ell-1}\{\varepsilon(((\ell^{2}+1)/2)(2a+\ell jb))b_{((\ell^{2}+1)/2)(2a+\ell jb),n} -\ell\sum_{u=0}^{\ell-2}b_{u\ell+((\ell^{2}+1)/2)(2a+\ell jb),n}\},
\end{align*}
where $b_{0}(n)=0$,\\
$\nu(b)= \begin{cases}
\varLambda(b)/2 \hspace*{8mm} \text{if } b \text{ is even};\\
\varLambda(b+\ell)/2 \hspace*{2mm} \text{if } b \text{ is  odd},
\end{cases}$
\\
and $\varLambda(r)$ is defined as the least non-negative residue of $r$ module $2\ell$.
\end{corollary}

For any odd prime $p\geq 5$, $p^2\equiv 1 \pmod 3$ holds. Thus to calculate cyclotomic numbers of order $2\ell^{2}$ with prime $\ell\geq 5$, it is sufficient to calculate $2\ell^2+(2\ell^2-3)+(2\ell^2-6)+(2\ell^2-9)+\cdots+2$ distinct cyclotomic numbers of order $2\ell^2$. However $\ell=3$, it is enough to calculate $2\ell^2+(2\ell^2-3)+(2\ell^2-6)+(2\ell^2-9)+\cdots +1$ distinct cyclotomic numbers of order $2\ell^2$. 

\section{Matrix assocated to cyclotomic numbers}
In this section, we illustrate two types of cyclotomic numbers. We first consider $q=p=19\equiv 1 \pmod{18}$ and $q=p=37\equiv 1 \pmod {18}$. 
Then using property \eqref{2.1} of cyclotomic numbers the $324$ pairs of two paremeters numbers $(a,b)_{18}$ can be reduced to $64$ distinct pairs (see Table \ref{table1} and Table \ref{table2}). On evaluating cyclotomic numbers corresponding to $64$ distinct pair, we obtain the complete tables in a form of matrices $A$ and $B$. Primary interest is to know about the determinant, eigen values, characteristic polynomial and minimal polynomial of $A$ and $B$. It is also important to know the nature of the matrices obtained by changing the generator of $\mathbb{F}_q^{\times}$. Since the entries of exactly one row of $A$ are zero, thus $det(A)=0$. We observe that the characteristic as well as minimal polynomial of $A$ is $m_A(x)=x^{18}$. We also see that all the eigenvalues of $A$ are equal, and in fact they all are zero.
Using GP/PARI (version 2.9.2), we obtain $det(B)=-1$ and the characteristic as well as minimal polynomial is \begin{align*}
m_B(x)&=x^{18}-x^{17}-17x^{16}+16x^{15}+120x^{14}-105x^{13}-455x^{12}+364x^{11}\\ & +1001x^{10}-715x^{9}-1287x^{8}+792x^{7}+924x^{6}-462x^{5}-330x^{4}
 +120x^{3}+45x^{2}-9x-1.
\end{align*}
The eigenvalues of $B$ are $\lambda_{1}=-1.9712,\ \lambda_{2}=1.9928,\ \lambda_{3}=1.9355,\ \lambda_{4}=-1.8858,\ \lambda_{5}=1.8225,\ \lambda_{6}=-1.7460,\ \lambda_{7}=1.6570,\ \lambda_{8}=-1.5561,\ \lambda_{9}=1.4439,\ \lambda_{10}=-1.3213,\ \lambda_{11}=1.1893, \lambda_{12}=-1.0486,\ \lambda_{13}=0.9004,\ \lambda_{14}=-0.7457, \lambda_{15}=0.5856,\ \lambda_{16}=-0.4214,\ \lambda_{17}=0.2540,\ \lambda_{18}=-0.0849$, and all of them are distinct. It is noted that if we change the generator of $\mathbb{F}_{q}^{\times}$, then entries of $A$ (resp. $B$) get interchange among themselves but their nature remain as the same.      
{\tiny
\begin{table}
\noindent\makebox[-5mm]{
\begin{tabular}{|p{0.5cm}|p{0.4cm} p{0.6cm} p{0.6cm} p{0.6cm} p{0.6cm} p{0.6cm} p{0.6cm} p{0.6cm} p{0.6cm} p{0.6cm} p{0.6cm} p{0.6cm} p{0.6cm} p{0.6cm} p{0.6cm} p{0.6cm} p{0.6cm} p{0.6cm}|}\hline(a,b) &\multicolumn{18}{c|}{b}\\ \hline
a & 0 & 1 & 2 & 3 & 4 & 5 & 6 & 7 & 8 & 9 & 10 & 11 & 12 & 13 & 14 & 15 & 16 & 17
\\ \hline
0 &(0,0)&(0,1)&(0,2)&(0,3)&(0,4)&(0,5)&(0,6)&(0,7)&(0,8)&(0,9)&(0,10)&(0,11)&(0,12)&(0,13)&(0,14)&(0,15)&(0,16)&(0,17)
\\ \hline
1 & (1,0) & (1,1) & (1,2) & (1,3) & (1,4) & (1,5) & (1,6) & (1,7) & (1,8) & (0,10) & (0,8) & (1,8) & (1,12) & (1,13) & (1,14) & (1,15) & (1,16) & (1,17)
\\ \hline
2 & (2,0) & (2,1) & (2,2) & (2,3) & (2,4) & (2,5) & (2,6) & (2,7) & (1,12) & (0,11) & (1,8) & (0,7) & (1,7) & (2,7) & (2,14) & (2,15) & (2,16) & (2,17)
\\ \hline
3 & (3,0) & (3,1) & (3,2) & (3,3) & (3,4) & (3,5) & (3,6) & (2,14) & (1,13) & (0,12) & (1,12) & (1,7) & (0,6) & (1,6) & (2,6) & (3,6) & (3,16) & (3,17)
\\ \hline
4 & (4,0) & (4,1) & (4,2) & (4,3) & (4,4) & (4,5) & (3,16) & (2,15) & (1,14) & (0,13) & (1,13) & (2,7) & (1,6) & (0,5) & (1,5) & (2,5) & (3,5) & (4,5)
\\ \hline
5 & (4,4) & (5,1) & (5,2) & (5,3) & (5,1) & (4,0) & (3,17) & (2,16) & (1,15) & (0,14) & (1,14) & (2,14) & (2,6) & (1,5) & (0,4) & (1,4) & (2,4) & (3,4)
\\ \hline
6 & (3,3) & (4,3) & (5,3) & (6,3) & (5,2) & (4,1) & (3,0) & (2,17) & (1,16) & (0,15) & (1,15) & (2,15) & (3,6) & (2,5) & (1,4) & (0,3) & (1,3) & (2,3)
\\ \hline
7 & (2,2) & (3,2) & (4,2) & (5,2) & (5,3) & (4,2) & (3,1) & (2,0) & (1,17) & (0,16) & (1,16) & (2,16) & (3,16) & (3,5) & (2,4) & (1,3) & (0,2) & (1,2)
\\ \hline
8 & (1,1) & (2,1) & (3,1) & (4,1) & (5,1) & (4,3) & (3,2) & (2,1) & (1,0) & (0,17) & (1,17) & (2,17) & (3,17) & (4,5) & (3,4) & (2,3) & (1,2) & (0,1)
\\ \hline
9 & (0,0) & (1,0) & (2,0) & (3,0) & (4,0) & (4,4) & (3,3) & (2,2) & (1,1) & (0,0) & (1,0) & (2,0) & (3,0) & (4,0) & (4,4) & (3,3) & (2,2) & (1,1)
\\ \hline
10 & (1,0) & (0,17) & (1,17) & (2,17) & (3,17) & (4,5) & (3,4) & (2,3) & (1,2) & (0,1) & (1,1) & (2,1) & (3,1) & (4,1) & (5,1) & (4,3) & (3,2) & (2,1)
\\ \hline
11 & (2,0) & (1,17) & (0,16) & (1,16) & (2,16) & (3,16) & (3,5) & (2,4) & (1,3) & (0,2) & (1,2) & (2,2) & (3,2) & (4,2) & (5,2) & (5,3) & (4,2) & (3,1)
\\ \hline
12 & (3,0) & (2,17) & (1,16) & (0,15) & (1,15) & (2,15) & (3,6) & (2,5) & (1,4) & (0,3) & (1,3) & (2,3) & (3,3) & (4,3) & (5,3) & (6,3) & (5,2) & (4,1)
\\ \hline
13 & (4,0) & (3,17) & (2,16) & (1,15) & (0,14) & (1,14) & (2,14) & (2,6) & (1,5) & (0,4) & (1,4) & (2,4) & (3,4) & (4,4) & (5,1) & (5,2) & (5,3) & (5,1)
\\ \hline
14 & (4,4) & (4,5) & (3,16) & (2,15) & 
(1,14) & (0,13) & (1,13) & (2,7) & (1,6) & (0,5) & (1,5) & (2,5) & (3,5) & (4,5) & (4,0) & (4,1) & (4,2) & (4,3)
\\ \hline
15 & (3,3) & (3,4) & (3,5) & (3,6) & (2,14) & (1,13) & (0,12) & (1,12) & (1,7) & (0,6) & (1,6) & (2,6) & (3,6) & (3,16) & (3,17) & (3,0) & (3,1) & (3,2)
\\ \hline
16 & (2,2) & (2,3) & (2,4) & (2,5) & (2,6) & (2,7) & (1,12) & (0,11) & (1,8) & (0,7) & (1,7) & (2,7) & (2,14) & (2,15) & (2,16) & (2,17) & (2,0) & (2,1)
\\ \hline
17 & (1,1) & (1,2) & (1,3) & (1,4) & (1,5) & (1,6) & (1,7) & (1,8) & (0,10) & (0,8) & (1,8) & (1,12) & (1,13) & (1,14) & (1,15) & (1,16) & (1,17) & (1,0)\\
\hline
\end{tabular}}
\vspace*{2mm}
\caption{Equality of $(a,b)_{18}$ if $k$ is odd} \label{table1}
\end{table}}

The matrix corresponding of Table \ref{table1} is given by 
\begin{center}
	
	$A$
	=$\left[
	\begin{array}{rrrrrrrrrrrrrrrrrr}
	0 & 0 & 0 & 0 & 0 & 0 & 0 & 0 & 0 & 0 & 0 & 1 & 0 & 0 & 0 & 0 & 0 & 0 
	\\
	0 & 0 & 0 & 0 & 0 & 1 & 0 & 0 & 0 & 0 & 0 & 0 & 0 & 0 & 0 & 0 & 0 & 0
	\\
	0 & 0 & 0 & 0 & 0 & 0 & 0 & 0 & 0 & 1 & 0 & 0 & 0 & 0 & 0 & 0 & 0 & 0
	\\
	0 & 1 & 0 & 0 & 0 & 0 & 0 & 0 & 0 & 0 & 0 & 0 & 0 & 0 & 0 & 0 & 0 & 0
	\\
	0 & 0 & 0 & 0 & 0 & 0 & 0 & 0 & 0 & 0 & 0 & 0 & 0 & 0 & 1 & 0 & 0 & 0
	\\
	0 & 0 & 0 & 0 & 0 & 0 & 0 & 0 & 0 & 0 & 0 & 0 & 0 & 1 & 0 & 0 & 0 & 0
	\\
	0 & 0 & 0 & 1 & 0 & 0 & 0 & 0 & 0 & 0 & 0 & 0 & 0 & 0 & 0 & 0 & 0 & 0
	\\
	0 & 0 & 0 & 0 & 0 & 0 & 1 & 0 & 0 & 0 & 0 & 0 & 0 & 0 & 0 & 0 & 0 & 0
	\\
    0 & 0 & 1 & 0 & 0 & 0 & 0 & 0 & 0 & 0 & 0 & 0 & 0 & 0 & 0 & 0 & 0 & 0
	\\
	0 & 0 & 0 & 0 & 0 & 0 & 0 & 0 & 0 & 0 & 0 & 0 & 0 & 0 & 0 & 0 & 0 & 0
	\\
	0 & 0 & 0 & 0 & 0 & 0 & 0 & 0 & 0 & 0 & 0 & 0 & 1 & 0 & 0 & 0 & 0 & 0
	\\
	0 & 0 & 0 & 0 & 0 & 0 & 0 & 0 & 0 & 0 & 0 & 0 & 0 & 0 & 0 & 0 & 0 & 1
	\\
	0 & 0 & 0 & 0 & 0 & 0 & 0 & 0 & 0 & 0 & 0 & 0 & 0 & 0 & 0 & 1 & 0 & 0
	\\
	0 & 0 & 0 & 0 & 0 & 0 & 0 & 0 & 1 & 0 & 0 & 0 & 0 & 0 & 0 & 0 & 0 & 0
	\\
	0 & 0 & 0 & 0 & 0 & 0 & 0 & 0 & 0 & 0 & 1 & 0 & 0 & 0 & 0 & 0 & 0 & 0
	\\
	0 & 0 & 0 & 0 & 0 & 0 & 0 & 0 & 0 & 0 & 0 & 0 & 0 & 0 & 0 & 0 & 1 & 0
	\\
	0 & 0 & 0 & 0 & 0 & 0 & 0 & 1 & 0 & 0 & 0 & 0 & 0 & 0 & 0 & 0 & 0 & 0
	\\
	0 & 0 & 0 & 0 & 1 & 0 & 0 & 0 & 0 & 0 & 0 & 0 & 0 & 0 & 0 & 0 & 0 & 0 \\
	\end{array}\right] 
	$ 
\end{center}
\newpage
\tiny{
\begin{table}
\noindent\makebox[-5mm]{
\begin{tabular}{|p{0.5cm}|p{0.4cm} p{0.6cm} p{0.6cm} p{0.6cm} p{0.6cm} p{0.6cm} p{0.6cm} p{0.6cm} p{0.6cm} p{0.6cm} p{0.6cm} p{0.6cm} p{0.6cm} p{0.6cm} p{0.6cm} p{0.6cm} p{0.6cm} p{0.6cm}|}\hline(a,b) &\multicolumn{18}{c|}{b}\\ \hline
a & 0 & 1 & 2 & 3 & 4 & 5 & 6 & 7 & 8 & 9 & 10 & 11 & 12 & 13 & 14 & 15 & 16 & 17
\\ \hline
0 &(0,0)&(0,1)&(0,2)&(0,3)&(0,4)&(0,5)&(0,6)&(0,7)&(0,8)&(0,9)&(0,10)&(0,11)&(0,12)&(0,13)&(0,14)&(0,15)&(0,16)&(0,17)
\\ \hline
1 & (0,1) & (0,17) & (1,2) & (1,3) & (1,4) & (1,5) & (1,6) & (1,7) & (1,8) & (1,9) & (1,10) & (1,11) & (1,12) & (1,13) & (1,14) & (1,15) & (1,16) & (1,2)
\\ \hline
2 & (0,2) & (1,2) & (0,16) & (1,16) & (2,4) & (2,5) & (2,6) & (2,7) & (2,8) & (2,9) & (2,10) & (2,11) & (2,12) & (2,13) & (2,14) & (2,15) & (2,4) & (1,3)
\\ \hline
3 & (0,3) & (1,3) & (1,16) & (0,15) & (1,15) & (2,15) & (3,6) & (3,7) & (3,8) & (3,9) & (3,10) & (3,11) & (3,12) & (3,13) & (3,14) & (3,6) & (2,5) & (1,4)
\\ \hline
4 & (0,4) & (1,4) & (2,4) & (1,15) & (0,14) & (1,14) & (2,14) & (3,14) & (4,8) & (4,9) & (4,10) & (4,11) & (4,12) & (4,13) & (4,8) & (3,7) & (2,6) & (1,5)
\\ \hline
5 & (0,5) & (1,5) & (2,5) & (2,15) & (1,14) & (0,13) & (1,13) & (2,13) & (3,13) & (4,13) & (5,10) & (5,11) & (5,12) & (5,10) & (4,9) & (3,8) & (2,7) & (1,6)
\\ \hline
6 & (0,6) & (1,6) & (2,6) & (3,6) & (2,14) & (1,13) & (0,12) & (1,12) & (2,12) & (3,12) & (4,12) & (5,12) & (6,12) & (5,11) & (4,10) & (3,9) & (2,8) & (1,7)
\\ \hline
7 & (0,7) & (1,7) & (2,7) & (3,7) & (3,14) & (2,13) & (1,12) & (0,11) & (1,11) & (2,11) & (3,11) & (4,11) & (5,11) & (5,12) & (4,11) & (3,10) & (2,9) & (1,8)
\\ \hline
8 & (0,8) & (1,8) & (2,8) & (3,8) & (4,8) & (3,13) & (2,12) & (1,11) & (0,10) & (1,10) & (2,10) & (3,10) & (4,10) & (5,10) & (4,12) & (3,11) & (2,10) & (1,9)
\\ \hline
9 & (0,9) & (1,9) & (2,9) & (3,9) & (4,9) & (4,13) & (3,12) & (2,11) & (1,10) & (0,9) & (1,9) & (2,9) & (3,9) & (4,9) & (4,13) & (3,12) & (2,11) & (1,10)
\\ \hline
10 & (0,10) & (1,10) & (2,10) & (3,10) & (4,10) & (5,10) & (4,12) & (3,11) & (2,10) & (1,9) & (0,8) & (1,8) & (2,8) & (3,8) & (4,8) & (3,13) & (2,12) & (1,11)
\\ \hline
11 & (0,11) & (1,11) & (2,11) & (3,11) & (4,11) & (5,11) & (5,12) & (4,11) & (3,10) & (2,9) & (1,8) & (0,7) & (1,7) & (2,7) & (3,7) & (3,14) & (2,13) & (1,12)
\\ \hline
12 & (0,12) & (1,12) & (2,12) & (3,12) & (4,12) & (5,12) & (6,12) & (5,11) & (4,10) & (3,9) & (2,8) & (1,7) & (0,6) & (1,6) & (2,6) & (3,6) & (2,14) & (1,13)
\\ \hline
13 & (0,13) & (1,13) & (2,13) & (3,13) & (4,13) & (5,10) & (5,11) & (5,12) & (5,10) & (4,9) & (3,8) & (2,7) & (1,6)  & (0,5) & (1,5) & (2,5) & (2,15) & (1,14) 
\\ \hline
14 & (0,14) & (1,14) & (2,14) & (3,14) & (4,8) & (4,9) & (4,10) & (4,11) & (4,12) & (4,13) & (4,8) & (3,7) & (2,6) & (1,5) & (0,4) & (1,4) & (2,4) & (1,15)
\\ \hline
15 & (0,15) & (1,15) & (2,15) & (3,6) & (3,7) & (3,8) & (3,9) & (3,10) & (3,11) & (3,12) & (3,13) & (3,14) & (3,6) & (2,5) & (1,4) & (0,3) & (1,3) & (1,16)
\\ \hline
16 & (0,16) & (1,16) & (2,4) & (2,5) & (2,6) & (2,7) & (2,8) & (2,9) & (2,10) & (2,11) & (2,12) & (2,13) & (2,14) & (2,15) & (2,4) & (1,3) & (0,2) & (1,2)
\\ \hline
17 & (0,17) & (1,2) & (1,3) & (1,4) & (1,5) & (1,6) & (1,7) & (1,8) & (1,9) & (1,10) & (1,11) & (1,12) & (1,13) & (1,14) & (1,15) & (1,16) & (1,2) & (0,1)\\
\hline
\end{tabular}}
\vspace*{2mm}
\caption{Equality of $(a,b)_{18}$ if $k$ is even}\label{table2}
\end{table}}

\large{
The matrix corresponding of Table \ref{table2} is given by 
\begin{center}
	
	$B$
	=$\left[
	\begin{array}{rrrrrrrrrrrrrrrrrr}
	0 & 0 & 0 & 0 & 0 & 0 & 0 & 0 & 0 & 0 & 0 & 1 & 0 & 0 & 0 & 0 & 0 & 0 
	\\
	0 & 0 & 0 & 0 & 1 & 0 & 0 & 0 & 0 & 1 & 0 & 0 & 0 & 0 & 0 & 0 & 0 & 0
	\\
	0 & 0 & 0 & 0 & 0 & 0 & 1 & 0 & 0 & 0 & 0 & 0 & 0 & 1 & 0 & 0 & 0 & 0
	\\
	0 & 0 & 0 & 0 & 0 & 0 & 0 & 0 & 0 & 0 & 0 & 0 & 0 & 1 & 0 & 0 & 0 & 1
	\\
	0 & 1 & 0 & 0 & 0 & 0 & 0 & 0 & 0 & 0 & 0 & 0 & 0 & 0 & 0 & 0 & 1 & 0
	\\
	0 & 0 & 0 & 0 & 0 & 0 & 0 & 1 & 1 & 0 & 0 & 0 & 0 & 0 & 0 & 0 & 0 & 0
	\\
	0 & 0 & 1 & 0 & 0 & 0 & 0 & 0 & 0 & 0 & 0 & 0 & 1 & 0 & 0 & 0 & 0 & 0
	\\
	0 & 0 & 0 & 0 & 0 & 1 & 0 & 1 & 0 & 0 & 0 & 0 & 0 & 0 & 0 & 0 & 0 & 0
	\\
    0 & 0 & 0 & 0 & 0 & 1 & 0 & 0 & 0 & 0 & 0 & 0 & 0 & 0 & 0 & 0 & 0 & 1
	\\
	0 & 1 & 0 & 0 & 0 & 0 & 0 & 0 & 0 & 0 & 1 & 0 & 0 & 0 & 0 & 0 & 0 & 0
	\\
	0 & 0 & 0 & 0 & 0 & 0 & 0 & 0 & 0 & 1 & 0 & 0 & 0 & 0 & 0 & 1 & 0 & 0
	\\
	1 & 0 & 0 & 0 & 0 & 0 & 0 & 0 & 0 & 0 & 0 & 0 & 0 & 0 & 0 & 0 & 1 & 0
	\\
	0 & 0 & 0 & 0 & 0 & 0 & 1 & 0 & 0 & 0 & 0 & 0 & 0 & 0 & 1 & 0 & 0 & 0
	\\
	0 & 0 & 1 & 1 & 0 & 0 & 0 & 0 & 0 & 0 & 0 & 0 & 0 & 0 & 0 & 0 & 0 & 0
	\\
	0 & 0 & 0 & 0 & 0 & 0 & 0 & 0 & 0 & 0 & 0 & 0 & 1 & 0 & 0 & 1 & 0 & 0
	\\
	0 & 0 & 0 & 0 & 0 & 0 & 0 & 0 & 0 & 0 & 1 & 0 & 0 & 0 & 1 & 0 & 0 & 0
	\\
	0 & 0 & 0 & 0 & 1 & 0 & 0 & 0 & 0 & 0 & 0 & 1 & 0 & 0 & 0 & 0 & 0 & 0
	\\
	0 & 0 & 0 & 1 & 0 & 0 & 0 & 0 & 1 & 0 & 0 & 0 & 0 & 0 & 0 & 0 & 0 & 0 \\
	\end{array}\right] 
	$ 
\end{center}}
It is noted if $k=1$, then the corresponding matrix is always singular. The Jordon-canonical form of $B$ is given by the diogonal matrix $D$ with diagonal
entries\\ 
$\displaystyle{ \lambda_{1},\lambda_{2},\lambda_{3},\lambda_{4},\lambda_{5},\lambda_{6},\lambda_{7},\lambda_{8}, \lambda_{9},\lambda_{10}, \lambda_{11}, 
\lambda_{12}, \lambda_{13},\lambda_{14},\lambda_{15},\lambda_{16},\lambda_{17},\lambda_{18}}$. 

\section*{Acknowledgments}
\noindent
M H Ahmed and J Tanti acknowledge Central University of Jharkhand, Ranchi, Jharkhand for providing necessary and excellent facilities to carry out this research. A  Hoque acknowledges SERB-NPDF (PDF/2017/001958), Govt. of India for financial support.

\end{document}